\documentclass[10pt,twoside]{siamltex}\usepackage{amssymb,enumerate}
\usepackage{amsmath,amssymb}
\usepackage{amsfonts,epsfig}

\newtheorem{thm}{Theorem}

\newtheorem{lem}{Lemma}
\newtheorem{prop}{Proposition}
\newtheorem{cor}{Corollary}
\newtheorem{rmk}{Remark}

\def\ex{{\noindent\bf Example. \ }}
\numberwithin{equation}{section}

\begin{document}



\bibliographystyle{plain}
\title{Wold decompositions for operators close to isometries}\thanks{Received by the editors on Month x, 200x.
Accepted for publication on Month y, 200y   Handling Editor: .}

\author{
Laura G\u avru\c ta\thanks{Fakult\"at fur Mathematik, University of Viena, Oskar-Morgenstern-Platz 1,
1090 Wien, Austria. \newline Actual adress: Politehnica University of Timi\c soara, Pia\c ta Victoriei, no. 2, Romania (gavruta\_laura@yahoo.com).
\newline The author was supported by the FWF project P 24986-N25
}}


\pagestyle{myheadings}
\markboth{L. G\u avru\c ta}{Wold decompositions for operators close to isometries}
\maketitle

\begin{abstract}
In Bergman and Dirichlet spaces, the shift operator is not an isometry, but it is a left invertible operator. In this paper we give conditions on the left invertible operators such that a operator version, in the sense of Rosenblum and Rovnyak, of the Wold decomposition to take place.
\end{abstract}


\section{Introduction} \label{intro-sec}

The Wold Decomposition Theorem \cite{Wold} applies to the analysis of stationary random processes. It provides a representation of such processes and also an interpretation of the representation in terms of linear prediction: an arbitrary unpredictable process can be written as an orthogonal sum of a regular process and a predictable process \cite{Wold}.

In 1961, Paul R. Halmos \cite{Halmos1} gives the following form of the Wold decomposition Theorem in operator language:\\

\begin{theorem}\label{Wthm}Let $V$ be an isometry on a Hilbert space $\mathcal{H}.$ Then there is a decomposition of $\mathcal{H}$ as a direct sum of two mutually orthogonal subspaces $$\mathcal{H}=\mathcal{H}_{\infty}\bigoplus\mathcal{H}_s$$ such that
\begin{enumerate}[($i$)]
\item $\mathcal{H}_{\infty}$ and $\mathcal{H}_s$ reduce $V$.
\item   The restriction of $V$ to $\mathcal{H}_\infty$ is an unitary operator.
\item The restriction of $V$ to $\mathcal{H}_s$ is unitarily equivalent to an unilateral shift.
\end{enumerate}
The decomposition is unique.
\end{theorem}

We recall that a subspace $\mathcal{H}_0$ of $\mathcal{H}$ reduce $V$ if $\mathcal{H}_0$ is invariant to $V$ and its adjoint.

In fact, the subspaces $\mathcal{H}_\infty$ and $\mathcal{H}_s$ are obtained in the following manner $$ \mathcal{H}_\infty=\bigcap\limits_{n=1}^\infty V^n\mathcal{H}$$
$$\mathcal{H}_s=\bigoplus\limits_{n=1}^{\infty}V^nW,$$ where $W:=\mathcal{H}\ominus V\mathcal{H}$ is the orthogonal complement of $V\mathcal{H}$ in $\mathcal{H}$. Here, $W$ is a wandering subspace for $V,$ that is $$V^mW\perp V^nW,\quad m\neq n.$$
See also \cite{Halmos2}, \cite{Shapiro}.\\
Theorem \ref{Wthm} has some remarkable consequences (see \cite{Bratelli}, \cite{Popoulis}) such as:
\begin{itemize}
\item the deduction of the Beurling's invariant subspace Theorem in Hardy spaces;
 \item  the description of the structure of a wide-sense stationary random sequence;
 \item  the description of the structure of isometric and unitary dilation spaces for contractions of a Hilbert space (see Nagy and Foia\c s \cite{Nagy}).
\end{itemize}

The above version of the Wold decomposition emphasizes spatial structure. An operator version of the Wold decomposition of an isometry is given by M. Rosenblum and J. Rovyak in their book \cite{Rosemblum}.

 But in Bergman and Dirichlet spaces the shift operator is no longer an isometry. Remarkable Wold type Theorem for classes of left invertible operators and applications to problems of invariant subspaces were obtained by S. Richter \cite{Richter} and S. Shimorin \cite{Shimorin}. In 1991, A. Aleman, S. Richter, C. Sundberg \cite{Aleman} proved the Beurling type theorem for Bergman shift, which was a big step in the study of invariant subspaces of the Bergman shift. This result became an important tool in the function theory of $L_a^2$ because it shows  the structure of invariant subspaces of the Bergman space.

This paper is motivated by a problem posed by S. Shimorin \cite{Shimorin}. The problem is to give new conditions for a left invertible operator to imply Wold type decompositions. The aim of this paper is to give conditions on the left invertible operators such that a operator version of the Wold decomposition can be proved. The left invertible operators (the operators bounded below) are the analysis operators from frame theory (see \cite{Christensen}).\\

We denote by $\mathcal{L}(\mathcal{H})$ the algebra of all linear bounded operators on the Hilbert space $\mathcal{H}$ and for $T\in\mathcal{L}{(\mathcal{H})}$, we denote by $T^*$ the adjoint operator of $T$. The following lemma is a well-known result.
\begin{lem} Let $T\in\mathcal{L}(\mathcal{H})$. The following are equivalent
\begin{enumerate}[($i$)]
\item $T$ is left invertible;
\item $T$ is bounded below, i.e. there exists a constant $m>0$ such that $$\|Th\|\geq m\|h\|,\quad h\in\mathcal{H};$$
\item $T^*$ is surjective;
\item $T^*T$ is invertible.
\end{enumerate}
\end{lem}
If $T$ is left invertible, then $T\mathcal{H}$ is a closed subspace of $\mathcal{H}.$ As in \cite{Richter, Shimorin} we distinguish the following left-inverse of $T$ $$T^{-}=(T^*T)^{-1}T^*$$
and its kernel $$W=\mathcal{H}\ominus T\mathcal{H}=Ker{T}^*.$$
The subspace $W$ is called the defect of $T.$ It is clear that if $T$ is left invertible then $T^n$ is left invertible.

In the following, $\mathcal{D}$ denotes the set of all left invertible operators on $\mathcal{H}$ for which the following condition holds
\begin{equation}\label{conditiontonT-}
(T^n)^{-}=(T^{-})^n,\quad \textrm{for all}~ n\geq 2.
\end{equation}
\begin{rmk} If $T$ is an isometry, then $T^*T=I$, hence $T^{-}=T^*$ and (\ref{conditiontonT-}) holds. In fact, if $T$ is left invertible and  $T^{-}=T^*$, then $T$ is an isometry. Indeed, from the relation $(T^*T)^{-1}T^*=T^*$ it follows $$(T^*T)^{-1}T^*T=T^*T$$ hence $I=T^*T.$
\end{rmk}

In the following we give conditions for operators to be in the class $\mathcal{D}.$
\begin{prop} Let $Q\in\mathcal{L}(\mathcal{H})$ be left invertible and quasinormal. Then $Q\in\mathcal{D}.$
\end{prop}
\begin{proof} Since $Q$ quasinormal, we have $(Q^*Q)Q=Q(Q^*Q).$

 We claim that this implies that ${Q^*}^nQ^n=(Q^*Q)^n,~\textrm{for all}~ n\geq 2$. We prove this by induction on $n$.

Indeed, for $n=2$ $${Q^*}^2Q^2=Q^*(Q^*Q)Q=(Q^*Q)(Q^*Q)$$
If ${Q^*}^nQ^n=(Q^*Q)^n$ then
\begin{align*}(Q^*Q)^{n+1}&=Q^*Q(Q^*Q)^n=(Q^*Q){Q^*}^nQ^n\\
                          &={Q^*}^n(Q^*Q)Q^n.\end{align*} We used the fact that $Q$ is quasinormal, hence $(Q^*Q)Q^*=Q^*(Q^*Q)$ and by induction, we have $(Q^*Q)Q^{*n}=Q^{*n}(Q^*Q).$\\
 \\
 So we get$(Q^*Q)^{n+1}={Q^*}^{n+1}Q^{n+1}.$\\
\\
It follows $$(Q^*Q)^{-n}=({Q^*}^nQ^n)^{-1}.$$
Hence
\begin{align*}
(Q^-)^n&=[(Q^*Q)^{-1}Q^*]^n={Q^*}^n(Q^*Q)^{-n}\\
       &={Q^*}^n({Q^*}^nQ^n)^{-1}=(Q^n)^{-}.
\end{align*}
\end{proof}

We recall that two operators $T_1$, $T_2$ are \textit{double commuting} if $T_1T_2$=$T_2T_1$ and $T_1T_2^*$=$T_2^*T_1$
\begin{prop} Let $T_1$, $T_2$ be  double commuting operators in $\mathcal{D}$. Then $$T_1T_2\in\mathcal{D}.$$
\end{prop}
\begin{proof} It is clear that $T_1T_2$ is left invertible since $$T_2^-T_1^-T_1T_2=T_2^-T_2=I.$$
From hypothesis, $$T_2(T_1^*T_1)=(T_1^*T_1)T_2$$
hence $$(T_1^*T_1)^{-1}T_2=T_2(T_1^*T_1)^{-1}.$$
It follows that
\begin{align*}
T_1T_2&=[(T_1T_2)^*(T_1T_2)]^{-1}(T_1T_2)^*\\
        &=(T_2^*T_2T_1^*T_1)^{-1}T_2^*T_1^*\\
        &=(T_1^*T_1)^{-1}(T_2^*T_2)^{-1}T_1^*T_2^*\\
        &=(T_1^*T_1)^{-1}T_1^*(T_2^*T_2)^{-1}T_2^*\\
        &=T_1^{-}T_2^{-}.
\end{align*}
From here we get
\begin{align*}
[(T_1T_2)^-]^n&=(T_1^-T_2^-)^n=(T_1^-)^n(T_2^-)^n\\
              &=(T_1^n)^-(T_2^n)^-=(T_1^nT_2^n)^-=[(T_1T_2)^n]^-,
\end{align*}
since $T_1^n, T_2^n\in\mathcal{D}$ and $T_1^n$ and $T_2^n$ are double commuting.
\end{proof}
\begin{cor} Let be $T_1\in\mathcal{D}$ and $T_2$ be normal and invertible, and $T_1T_2=T_2T_1.$ Then $T_1T_2\in\mathcal{D}.$
\end{cor}
\begin{proof} From the hypothesis it follows that  $T_2\in\mathcal{D}$. From the Fuglede-Putnam theorem \cite{Halmos2} it follows $T_1^*T_2=T_2T_1^*$. The conclusion now is a consequence of Proposition 2.  \end{proof}
\begin{rmk} From the above corollary, we also obtain that each quasinormal and left invertible operator is in $\mathcal{D}.$
\end{rmk}
Indeed, $T$ admits polar decomposition $T=VA$, with $V$ isometry, $A=(T^*T)^{1/2}$ and $VA=AV$ (see \cite{Furuta}).\\

Next, we give an example of quasinormal, left invertible operator that is no an isometry.\\

\ex Let ${\cal K}$ be a Hilbert space of dimension at least 2 and $$l^2(\mathcal{K})=\{\widetilde{k}=(k_0,k_1,\ldots,k_n,\ldots)|k_j\in\mathcal{K},j=0,1,\ldots~\textrm{and}~ \|\widetilde{k}\|_2^2:=\sum_{j=0}^{\infty}\|k_j\|^2<\infty\}.$$
Let $L\in\mathcal{L}(\mathcal{K})$ be a positive invertible operator such that $\|Lk\|\geq m\|k\|,$ for all $k$ and some $m>1.$ We define the following operator on $l^2(\mathcal{K})$:
$$T\widetilde{k}:=(0,Lk_0,Lk_1,\ldots)$$
Note that $T$ is bounded on $l^2(\mathcal{K})$ and is not surjective;
$$(T^*T\widetilde{k})_n=L^2k_n$$
$$[T(T^*T)\widetilde{k}]=(0, L^3k_0,\ldots,L^3k_n,\ldots)$$
$$(T^*T)T\widetilde{k}=T^*T(0,Lk_0,\ldots,Lk_n,\ldots)=(0,L^3k_0,\ldots, L^3k_n,\ldots)$$
It follows that $T$ is quasinormal.\\
We have $\|T\widetilde{k}\|_2^2=\displaystyle \sum \|Lk_n\|^2\geq m^2\sum\|k_n\|^2.$\\
This implies that $T$ is not an isometry, since $\|T\widetilde{k}\|_2^2\neq\sum\|k_n\|^2.$

We give conditions for the weighted shifts \cite{Shields} and weighted translation operators \cite{Embry} to be in $\mathcal{D}.$
\begin{prop}Every bounded left invertible unilateral weighted shift on $l^2$ is in the class $\mathcal{D}.$
\end{prop}
\begin{proof} Let $T$ be a unilateral weighted shift, which is bounded below; that is, $Te_k=w_ke_{k+1},\quad k\geq 0$ and $C_1\leq w_n\leq C_2$, where $C_1, C_2$ are positive constants. We have:\\
$$T^*e_k =
\begin{cases}
\overline{w}_{k-1}e_{k-1}, \quad k\geq 1 \\
0, \quad k=0;
\end{cases}
$$
$$T^ne_k=w_kw_{k+1}\cdots w_{k+n-1}e_{k+n},~\textrm{and}$$
$${T^*}^ne_k =
\begin{cases}
\overline{w}_{k-1}\overline{w}_{k-2}\cdots\overline{w}_{k-n}e_{k-n}, \quad k\geq n \\
0, \quad 0\leq k<n.
\end{cases}
$$
It follows \begin{align*}
(T^-)^*e_k&=T(T^*T)^{-1}e_k=T\bigg(\dfrac{1}{|w_k|^2}e_k\bigg)\\
&=\dfrac{1}{|w_k|^2}w_ke_{k+1}\\
&=\dfrac{1}{\overline{w}_k}e_{k+1}
\end{align*}
and $${(T^-)^*}^ne_k=\dfrac{1}{\overline{w}_k\overline{w}_{k+1}\cdots\overline{w}_{k+n-1}}e_{k+n}.$$
On the other hand, $${T^*}^nT^ne_k=w_kw_{k+1}\cdots w_{k+n-1}\overline{w}_{k+n-1}\cdots\overline{w}_ke_k$$
which implies \begin{align*}
[(T^n)^-]^*&=T^n\bigg(\frac{1}{|w_k|^2\cdots|w_{k+n-1}|^2}\bigg)e_k\\
           &=\frac{1}{\overline{w}_k\cdots\overline{w}_{k+n-1}}e_{k+n}
\end{align*}
It follows $(T^-)^n=(T^n)^-,~\textrm{for all}~n\geq 2.$
\end{proof}

From the above Proposition, it follows that the Bergman shifts, i.e. the shifts with sequence weights $\bigg\{\sqrt{\dfrac{k+1}{k+2}}\bigg\}_{k\in\mathbb{N}}$ and also, the Dirichlet shifts, i.e. the shifts with sequence weights $\bigg\{\sqrt{\dfrac{k+2}{k+1}}\bigg\}_{k\in\mathbb{N}}$ are in the class of $\mathcal{D}.$
\begin{prop}Every left invertible weighted translation operator on $L^2(0,\infty)$ is in $\mathcal{D}.$
\end{prop}
\begin{proof} Let $T$ be a weighted translation operator, i.e.
$$Tf(x)=
\begin{cases}
\dfrac{\varphi(x)}{\varphi(x-t)}f(x-t),~\textrm{if}~x>t \\
0,~\textrm{if}~0<x\leq t
\end{cases}
$$
 Further we have
$$T^*f(x)=\frac{\varphi(x+t)}{\varphi(x)}f(x+t)$$
$$T^nf(x)=\frac{\varphi(x)}{\varphi(x-nt)}f(x-nt),~\textrm{for}~ x>nt$$ and
$${T^*}^nf(x)=\frac{\varphi(x+nt)}{\varphi(x)}f(x+nt).$$ It follows
that $${T^*}^nT^nf(x)=\bigg[\frac{\varphi(x+nt)}{\varphi(x)}\bigg]^2f(x),$$
$$T^*Tf(x)=\bigg[\frac{\varphi(x+t)}{\varphi(x)}\bigg]^2f(x),$$
$$T^-f(x)=\frac{\varphi(x)}{\varphi(x+t)}f(x+t),$$
$$(T^-)^nf(x)=\frac{\varphi(x)}{\varphi(x+nt)}f(x+nt),~\textrm{and}$$
$$(T^n)^-f(x)=\frac{\varphi(x)}{\varphi(x+nt)}f(x+nt).$$\end{proof}\\
Next, we give the main result of this paper. We recall here the following notation $T^-=(T^*T)^{-1}T^*.$
\begin{thm}Let $T\in\mathcal{L}(\mathcal{H})$ be in $\mathcal{D}.$ Then
\begin{enumerate}[($i)$]
\item $P_0:=I-TT^-$ is the projection of $\mathcal{H}$ on $W:=\mathcal{H}\ominus T\mathcal{H}$;
\item as $n\rightarrow\infty,$ $T^n(T^-)^n$ converges strongly to the projection operator, P, on $$
\bigcap_{n=1}^{\infty}T^n\mathcal{H};$$
\item $\displaystyle\sum_{j=0}^{\infty} T^jP_0(T^-)^j$ converges strongly to $Q:=I-P;$\\
\item $Q\mathcal{H}=\{h\in\mathcal{H}: \displaystyle \lim_{n\rightarrow\infty}\|({T^*}^nT^n)^{-1/2}{T^*}^n\|=0 \}$;\\
\item $P\mathcal{H}$ and $Q\mathcal{H}$ reduce $T$;
\item $T|_{P\mathcal{H}}$ is surjective;
\item $\displaystyle I=P+\sum_{j=0}^{\infty}T^jP_0(T^-)^j$
\end{enumerate}
\end{thm}
\begin{proof}Let $P_n=T^n(T^-)^n,$ $n\geq 1$. We prove that $P_n$ is the orthogonal projection of $\mathcal{H}$ on $T^n\mathcal{H},$ $n\geq 1.$ Indeed,
$$P_1^2=T(T^-T)T^-=TT^-=P_1$$
$$P_1^*=(T^-)^*T^*=T(T^*T)^{-1}T^*=TT^-=P_1.$$
Hence $P_1$ is the orthogonal projection of $\mathcal{H}$ on $P_1\mathcal{H}=TT^-\mathcal{H}=T\mathcal{H}$ since $T^-$ is surjective. For $n\geq 2,$ $T^n$ is also left invertible.

From the above result it follows that $P_n=T^n(T^n)^-=T^n(T^-)^n$ is the orthogonal projection of $\mathcal{H}$
on $T^n\mathcal{H}.$

It is clear that $P_0:I-TT^-$ is the orthogonal projection of $\mathcal{H}$ on $\mathcal{H}\ominus T\mathcal{H}.$
We prove that $P_nh\rightarrow Ph,~\textrm{for all}~h\in\mathcal{H}$ and $T^m W\perp T^nW,\quad m\neq n.$\\
It is clear that $P_n-P_{n+1}$ is the orthogonal projection of $\mathcal{H}$ on $T^n\mathcal{H}\bigcap(T^{n+1}\mathcal{H})^{\perp}.$

It follows $$(P_n-P_{n+1})\mathcal{H}\perp(P_m-P_{m+1})\mathcal{H}, m\neq n.$$
Hence
\begin{align*}\sum_{n=0}^m\|P_nh-P_{n+1}h\|^2&=\|\sum_{n=0}^m(P_nh-P_{n+1}h)\|^2\\
                                             &=\|h-P_{m+1}h\|^2\leq 4\|h\|^2
\end{align*}
It follows that $\displaystyle \sum_{n=0}^{\infty}\|P_nh-P_{n+1}h\|^2$ converges, i.e. for every  $\varepsilon>0$, there exists an  $N(\varepsilon)$ such that for $n\geq N(\varepsilon),$ we have:
\begin{align*}\|P_n h-P_{n+1}h\|^2&+\|P_{n+1}h-P_{n+2}h\|^2+\ldots+\|P_{n+p-1}h-P_{n+p}h\|^2<
\varepsilon\\
&\Longleftrightarrow\|P_nh-P_{n+p}h\|^2<\varepsilon,~\textrm{for all}~p\in\mathbb{N}, p\geq 1.
\end{align*}
So $(P_nh)$ converges to an element in $\mathcal{H}.$\\

We denote $\displaystyle Ph=\lim_{n\rightarrow\infty}P_nh.$ We prove that $P$ is the orthogonal projection of $\displaystyle\mathcal{H}$ on $\displaystyle\bigcap_{n=1}^{\infty} T^n\mathcal{H}$.
\begin{enumerate}[$\noindent$]
\item Indeed, we consider $\displaystyle h\in\bigcap_{n=1}^{\infty} T^n\mathcal{H}$. Then $h\in T^n\mathcal{H},~\textrm{for all}~n\geq 1$ and $P_n h=h,$ for all $n\geq 1$. Hence $Ph=h.$
\item On the other hand, if we take $\displaystyle h\perp\bigcap_{n=1}^{\infty} T^n\mathcal{H}$. Notice that, by the definition of $P$, it follows that $\displaystyle Ph\in\bigcap_{n=1}^{\infty} T^n\mathcal{H}.$ Then $$P(Ph)=\lim_{n\rightarrow\infty} P_n(Ph)=\lim_{n\rightarrow\infty} Ph=Ph$$
    and hence $$\|Ph\|^2=\langle Ph,Ph\rangle=\langle P^2h,h\rangle=\langle Ph,h\rangle=0.$$ Thus $Ph=0.$\end{enumerate}
We have \begin{align*}
P_n-P_{n+1}&=T^n(T^-)^n-T^{n+1}(T^-)^{n+1}\\
           &=T^n(I-TT^-)(T^-)^n\\
           &=T^nP_0(T^-)^n
\end{align*}
and $$(I-P_1)+(P_1-P_2)+\ldots+(P_n-P_{n+1})=I-P_{n+1}.$$
Hence $\displaystyle \sum_{n=0}^{\infty}T^nP_0(T^-)^{n}$ converges to $I-P=Q.$\\
\\
To prove $(iv)$, we observe that
$$h\in Q\mathcal{H}\Longleftrightarrow Ph=0\Longleftrightarrow\displaystyle \lim_{n\rightarrow\infty}\|T^n(T^-)^nh\|=0.$$
The last equality is equivalent with $$\lim_{n\rightarrow\infty}\langle T^n({T^*}^nT^n)^{-1}{T^*}^nh,T^n({T^*}^n T^n)^{-1}{T^*}^{n}h\rangle=0$$ i.e. $\displaystyle \lim_{n\rightarrow\infty}\|({T^*}^nT^n)^{-1/2}{T^*}^nh\|=0.$\\
\\
To prove that $\mathcal{H}_{\infty}$ reduces $T$ we note that
\begin{align*}
P_{n+1}T&=T^{n+1}(T^-)^{n+1}T=T^{n+1}(T^-)^n(T^-T)=T^{n+1}(T^-)^n\\
TP_n&=TT^n(T^-)^n=T^{n+1}(T^-)^n
\end{align*}
Hence $P_{n+1} T=TP_n\Rightarrow PT=TP$ hence $\mathcal{H}_{\infty}$ reduces $T.$\\
\\
Next we prove now that $T|_{\mathcal{H}_{\infty}}$ is surjective.\\
\\
Let $h_0\in\mathcal{H}_{\infty}$. It follows $h_0\in T^n\mathcal{H},~\textrm{for all}~n\geq 1.$\\
\\ For any $n\geq 1$ there exists $h_n\in\mathcal{H}$ so that $h_0=T^nh_n.$ Then $$h_0=Th_n', h_n'\in T^{n-1}\mathcal{H}, n\geq 1\Rightarrow T^-h_0=h_n', n\geq 1\Rightarrow h_n'=h_1',n\geq 1.$$
Hence $$h_0=Th_1', h_1'\in T^{n-1}\mathcal{H}, n\geq 1\Rightarrow h_1'\in\displaystyle \bigcap_{n\geq 1} T^{n-1}\mathcal{H}=\mathcal{H}_{\infty}.$$
\end{proof}
\begin{thm} Let $T\in\mathcal{D}.$ Then $W:=\mathcal{H}\ominus T\mathcal{H}$ is a wandering subspace of $\mathcal{H}$ and $$\mathcal{H}=\mathcal{H}_{\infty}\oplus\mathcal{H}_s,$$ where $$\mathcal{H}_{\infty}=\bigcap_{n=1}^{\infty}T^n\mathcal{H},\quad\mathcal{H}_s=\bigoplus_{n=0}^{\infty}T^nW.$$
$\mathcal{H}_{\infty}$ and $\mathcal{H}_s$ are reducing spaces of $T$ and $T|_{\mathcal{H}_{\infty}}$ is bijective. The decomposition is unique.
\end{thm}
\begin{proof}
The fact that the decomposition exists follows from Theorem 1. We prove that the decomposition is unique. Let $\mathcal{H}=\mathcal{H}_{\infty}'\bigoplus\mathcal{H}_s'$ a decomposition such that $$T\mathcal{H}_{\infty}'=\mathcal{H}_{\infty}'$$
$$\mathcal{H}_s'=\bigoplus_{n=0}^{\infty}T^nW',$$ where $W'$ is a wandering subspace of $T$.
We prove that $\mathcal{H}_{\infty}'=\mathcal{H}_{\infty}$ and $\mathcal{H}_s'=\mathcal{H}_s.$ Indeed,
\begin{align*}
W&=\mathcal{H}\ominus T\mathcal{H}=(\mathcal{H}_{\infty}'\bigoplus\mathcal{H}_s')\ominus(T\mathcal{H}_{\infty}'\bigoplus T\mathcal{H}_s')\\ &=(\mathcal{H}_{\infty}'\bigoplus\mathcal{H}_s')\ominus(\mathcal{H}_{\infty}'\bigoplus T\mathcal{H}_s')\\
&=\mathcal{H}_s'\ominus T\mathcal{H}_s'=W'
\end{align*}
We use the following fact:\\
If $\mathcal{H}=\mathcal{H}_1\oplus\mathcal{H}_2$ is such that $\mathcal{H}_1$, $\mathcal{H}_2$ are reducing subspaces of $T$, then $$T\mathcal{H}=T\mathcal{H}_1\oplus T\mathcal{H}_2.$$
This is clear because $$h=h_1+h_2,\quad h_1\perp h_2,\quad h_1\in\mathcal{H}_1, h_2\in\mathcal{H}_2$$
$Th=Th_1+Th_2$ and $\langle Th_1, Th_2\rangle=\langle T^*Th_1, h_2\rangle=0.$
\end{proof}

The following result is a Wold-type decomposition for a pair of double commuting operators in $\mathcal{D}$.
\begin{thm}Let be $T_1, T_2\in\mathcal{D}$ double commuting. Then $\mathcal{H}$ has the following orthogonal decomposition $$\mathcal{H}=\mathcal{H}_{\infty\infty}\bigoplus\mathcal{H}_{\infty s}\bigoplus\mathcal{H}_{s\infty}\bigoplus\mathcal{H}_{ss},$$ where $\mathcal{H}_{\infty\infty},$ $\mathcal{H}_{\infty s} ,$ $\mathcal{H}_{s\infty},$ $\mathcal{H}_{ss}$ are reducing spaces of $T_i (i=1,2)$ and
\begin{align*}
\mathcal{H}_{\infty\infty}&=\bigcap_{m,n=1}^{\infty} T_1^mT_2^n\mathcal{H}\\
\mathcal{H}_{\infty s} &=\bigg(\bigcap_{m=1}^{\infty} T_1^m\mathcal{H}\bigg)\bigcap\bigg(\bigoplus_{n=0}^{\infty}T_2^nW_2\bigg)\\
\mathcal{H}_{s\infty}&=\bigg(\bigoplus_{m=0}^{\infty}T_1^mW_1\bigg)\bigcap\bigg(\bigcap_{n=1}^{\infty} T_2^n\mathcal{H}\bigg)\\
\mathcal{H}_{ss}&=\bigg(\bigoplus_{m=0}^{\infty}T_1^mW_1\bigg)\bigcap\bigg(\bigoplus_{n=0}^{\infty} T_2^nW_2\bigg).
\end{align*}
\end{thm}
\begin{proof}
We denote by $Q_i$ the orthogonal projection on $\displaystyle\bigcap_{n=1}^{\infty}T_i^n\mathcal{H},\hspace{2mm}(i=1,2)$. From hypothesis, $Q_1$, $Q_2$ are commuting. The decomposition given in  Theorem 3 follows from the identity $$I=Q_1Q_2+Q_1(I-Q_2)+(I-Q_1)Q_2+(I-Q_1)(I-Q_2).$$\end{proof}

\bigskip
{\bf Acknowledgment.} I would like to thank Dr. Olivia Constantin for
introducing me in the theory of Bergman spaces and for her useful remarks. Also, I thank the referee for the careful reading of the paper and his/her valuable and useful comments  regarding our paper.



\begin{thebibliography}{1}
\bibitem{Aleman} A. Aleman, S. Richter, C. Sundberg.
\newblock Beurling's theorem for the Bergman space.
 \newblock {\em Acta Math.} 117:275--310, 1996.

\bibitem{Bratelli} O. Bratteli, P.E.T. Jorgensen.
 \newblock Isometries, shifts, Cunz algebras and multiresolution wavelet analysis of scale $N$.
  \newblock {\em Integr. equ. oper. theory} 28:382--443, 1977.

\bibitem{Christensen} O. Christensen.
\newblock {\em An Introduction to Frames and Riesz Bases}.
\newblock Birkha\" user, 2003.

\bibitem{Embry} M.R. Embry, A. Lambert.
 \newblock Weighted translation semigroups.
 \newblock {\em Rocky Mountain J. Math.}  7(2):333--344, 1977.

\bibitem{Furuta} T. Furuta.
\newblock On the polar decomposition of an operator.
\newblock{\em Acta Sci. Math. (Szeged)} 46:261--268, 1983.

\bibitem{Halmos1} P.R. Halmos.
\newblock Shifts on Hilbert spaces.
\newblock {\em J. Reine Angew. Math.} 208:102--112, 1961.

\bibitem{Halmos2}  P.R. Halmos.
\newblock {\em A Hilbert Space Problem Book}.
\newblock Van Nostrand, Princeton, N.J., 1967.

\bibitem{Popoulis} A. Papoulis.
 \newblock Predicable Procesess and Wold's Decomposition.
  \newblock {\em IEEE Trans. Acoustic, Speech, Signal Proc.} 33:933--938, 1985.

\bibitem{Richter} S. Richter.
 \newblock Invariant subspaces of the Dirichlet shift.
 \newblock {\em J. Reine Angerw. Math.} 386:205--220, 1988.

\bibitem{Rosemblum} M. Rosenblum, J. Rovnyak.
\newblock {\em Hardy Classes and Operator Theory}.
\newblock Oxford University Press, 1985.

\bibitem{Shapiro} H.S. Shapiro.
 \newblock {\em Operator Theory and Harmonic Analysis-A Celebration}.
  \newblock Proc. NATO Advanced Study Institute, held in Ciocco, Italy, 2--15 July 2000 (ed. J.S. Byrnes) NATO SCIENCE SERIES: II: Mathematics, Physics, Chemistry 33, Kluwer Academic Publishers, 31--56,  2001.

\bibitem{Shields} A. Shields.
 \newblock Weighted shift operators and analytic function theory.
 \newblock {\em  Amer. Math. Soc. Surveys} 13:49--128 (1974).

\bibitem{Shimorin} S. Shimorin.
 \newblock Wold-type decompositions and wandering subspaces for operators close to isometries.
  \newblock {\em J. Reine Angew. Math.} 531:147--189, 2001.

\bibitem{Nagy} S.Sz.-Nagy, C. Foia\c s.
 \newblock {\em Analyse Harmonique des Op\'erateurs de L'Espace de Hilbert, Masson et $C^{ie}$}.
  \newblock Akad\'emiai Kiad\'u, Budapest, 1967.

\bibitem{Wold} H. Wold.
\newblock {\em A study in the analysis of Stationary Time Series}.
\newblock Uppsala, Stockholm, 1938.

\end{thebibliography}
\end{document}